\RequirePackage{fix-cm}
\documentclass[smallextended]{svjour3}       
\smartqed  

\usepackage{graphicx}
\usepackage{calc,amsfonts,amssymb,amsmath,bm,url,color,cite}
\usepackage{graphicx,subfigure}
\usepackage{epsfig,multirow}
\usepackage{algorithmic,algorithm}
\usepackage{epstopdf}
\usepackage{comment}

\usepackage{algorithm}
\usepackage{algorithmic}

\newtheorem{lem}{Lemma}
\newtheorem{thm}{Theorem}

\newtheorem{defi}{Definition}

\newtheorem{ass}{Assumption}

\newcommand{\limto}{\rightarrow}

\begin{document}

\title{Phase Retrieval via Sensor Network Localization
}


\author{Sherry Xue-Ying Ni \and Man-Chung Yue \and \mbox{Kam-Fung Cheung \and Anthony Man-Cho So}
}

\institute{Sherry Xue-Ying Ni  \at
              Department of Systems Engineering and Engineering Management,\\
The Chinese University of Hong Kong\\
              Tel.: +852~3943-8313\\
              Fax: +852~2603-5505\\
              \email{xyni@se.cuhk.edu.hk}           
           \and
           Man-Chung Yue \at
           Imperial College Business School, \\
           Imperial College London, United Kingdom\\
           \email{m.yue@imperial.ac.uk}
           \and
           Kam-Fung Cheung \at
           Department of Systems Engineering and Engineering Management, \\
           The Chinese University of Hong Kong\\
           \email{kfcheung@link.cuhk.edu.hk}
           \and
           Anthony Man-Cho So \at
            Department of Systems Engineering and Engineering Management, \\
           The Chinese University of Hong Kong\\
           \email{manchoso@se.cuhk.edu.hk}           
}

\date{Received: date / Accepted: date}

\maketitle

\begin{abstract}
The problem of phase retrieval is revisited and studied from a fresh perspective. In particular, we establish a connection between the phase retrieval problem and the sensor network localization problem, which allows us to utilize the vast theoretical and algorithmic literature on the latter to tackle the former. Leveraging this connection, we develop a two-stage algorithm for phase retrieval that can provably recover the desired signal. In both sparse and dense settings, our proposed algorithm improves upon prior approaches simultaneously in the number of required measurements for recovery and the reconstruction time. We present numerical results to corroborate our theory and to demonstrate the efficiency of the proposed algorithm. As a side result, we propose a new form of phase retrieval problem and connect it to the complex rigidity theory proposed by Gortler and Thurston~\cite{gortler2014generic}.
\end{abstract}

\section{Introduction}
\subsection{Phase Retrieval Problem}
The problem of phase retrieval consists of recovering a signal vector $x\in \mathbb{C}^n$ from phaseless intensity measurements of the form 
\begin{equation}\label{eq:measurement}
|\langle x, \phi_m\rangle|^2=b_m, ~{m=1,2,\ldots,M},
\end{equation}
where for each $m=1,\dots,M$, $b_m\in\mathbb{R}_+$ is the observed output of the intensity measurement associated with a given measurement vector $\phi_m\in\mathbb{C}^n$. A collection $\Phi = \{\phi_m \}_{m=1}^M$ of measurement vectors is called an ensemble. Throughout the paper, we focus on the setting where we can freely design these measurement vectors $\Phi$. As we will see, the design of the ensemble $\Phi$ is of utmost importance to the recovery procedure. Note that for any unit-modulus complex number $e^{i\theta}$, the vector $e^{i\theta}x$ yields the same measurements. Therefore, we could recover the signal $x$ only up to the equivalence relation $\sim$ given by 

$$x\sim y~{\rm if}~{\rm and}~{\rm only}~{\rm if}~ y = e^{i\theta} x~{\rm for}~{\rm some}~ \theta\in \mathbb{R}.$$ 
Let $\mathbb{C}^n\!/\!\! \sim$ be the set of equivalence classes induced by the equivalence relation $\sim$ and denote by $\mathcal{A}_{\Phi}:\mathbb{C}^n\!/\!\! \sim\rightarrow \mathbb{R}_+$ the intensity map associated with the ensemble $\Phi = \{\phi_m \}_{m=1}^M$; i.e., \begin{equation}
\left(\mathcal{A}_{\Phi}(x) \right)_m = |\langle x, \phi_m\rangle|^2,\quad m=1,\dots,M.
\end{equation}
For simplicity, we will write $\mathcal{A}$ in place of $\mathcal{A}_\Phi$ when the ensemble $\Phi$ is clear from the context.

The phase retrieval problem has a long history and received great attention due to its vast modelling power in many areas. Fields of applications include X-ray and crystallography imaging \cite{harrison1993phase}, quantum optics \cite{mirhosseini2014compressive}, astronomy \cite{fienup1987phase}, acoustics \cite{balan2010signal}, and microscopy \cite{miao2008extending}. For more discussions on the history, applications, and recent developments of phase retrieval, we refer the readers to the excellent surveys~\cite{L17,SECCMS15}.

\subsection{Related Work}
Over the past few decades, the phase retrieval problem has been extensively studied in the literature. A popular approach in practice is to use the so-called error reduction-type algorithms. Algorithms that fall into this class, including the famous Gerchberg-Saxton algorithm~\cite{gerchberg1972practical} and the Fienup algorithm~\cite{fienup1978reconstruction,fienup1982phase,fienup1987phase}, are essentially alternating projection-type algorithms \cite{L17}. The advantage of this approach is its relatively low computational complexity and flexibility in incorporating prior knowledge of the signal into the recovery process. Unfortunately, this approach often lacks provable convergence guarantees and suffers from the issues of multiple stationary points and instability, especially with non-convex priors \cite{L17}. Another weakness is that the number of measurements required when using these algorithms is not known a priori, though some efforts to remedy this have been made in \cite{netrapalli2013phase}. In particular, Netrapalli et al. \cite{netrapalli2013phase} studied a version of the alternating projection method for the phase retrieval problem. Their algorithm recovers the signal using $\mathcal{O}(n \log ^3 n \log(1/\epsilon))$ intensity measurements and has computational complexity $\mathcal{O}(n^2\log n (\log^2 n+\log^2\frac{1}{\epsilon}\log \log \frac{1}{\epsilon}))$. However, this is still far from explaining the empirical success of error reduction-type algorithms and a rigorous mathematical foundation for this approach remains elusive.

Another recent approach is based on semidefinite programming and convex relaxation. The basic idea of this approach is to interpret quadratic measurements~\eqref{eq:measurement} as linear measurements of a rank-one matrix ${X}={x}{x}^{H}$. Then, the phase retrieval problem can be equivalently rewritten as a rank-minimization problem. Subsequently, by using a convex surrogate such as the trace norm to replace the rank function, we obtain a semidefinite program that can be solved in polynomial-time by off-the-shelf solvers. PhaseLift proposed by Cand\`es et al. \cite{candes2013phaselift} and PhaseCut by Waldspurger et al. \cite{waldspurger2015phase} are examples of such an approach. The drawback of this approach is its high computational complexity. Indeed, the complexities of PhaseLift and PhaseCut to return a solution of $\epsilon$ accuracy are $\mathcal{O}(n^3/\epsilon^2)$ and $\mathcal{O}(n^3/\sqrt{\epsilon})$, respectively. Both methods use an ensemble of 
$\mathcal{O}(n \log n)$ i.i.d. standard $n$-dimensional Gaussian random vectors. Assuming the signal $x$ is $s$-sparse (i.e., $x$ has at most $s$ non-zero components), the $\ell_1$-regularized version of PhaseLift \cite{li2013sparse} improves the number of required Gaussian measurements to $\mathcal{O}(s^2\log n)$. Nonetheless, this algorithm again requires solving a semidefinite program and hence has a similar computational complexity as PhaseLift and PhaseCut. Therefore, this approach is not applicable to large-scale phase retrieval problems in practice.

Other approaches usually involve construction of special matrices. In \cite{iwen2015fast}, Iwen et al. constructed block circulant measurement matrices that can be block diagonalized. By constructing certain invertible block circulant matrices, one can express the available squared magnitudes as a system of linear measurements, thereby recovering the signal. The approach reduces computational complexity to $\mathcal{O}(n(\log^3 n\log^3(\log n)))$; whereas the ensemble size is still as large as $\mathcal{O}(n\log^2 n\log^3(\log n))$ to guarantee unique recovery with high probability. In a recent work \cite{candes2015phase}, a non-convex approach based on Wirtinger flow was introduced to extract phase information from fewer random measurements. The number of measurements and complexity of this algorithm are both $\mathcal{O}(n\log n)$. In \cite{pedarsani2017phasecode}, the authors studied the general compressive phase retrieval problem with sparsity $s$. They developed a novel approach based on a sparse-graph coding framework and can recover a random fraction of non-zero components with $14s$ measurements and complexity $\Theta(s)$. Nevertheless, their method is only capable of correctly recovering part of the non-zeros entries and they still require $4s-o(s)$ measurements.

There have also been endeavours to understand the minimum size of an ensemble so that the measurements uniquely determine the signal up to the equivalence relation; i.e., the intensity map $\mathcal{A}$ is injective. Towards that end, Bandeira et al. \cite{bandeira2014saving} conjectured that $4n-4$ generic measurement vectors are both necessary and sufficient for the injectivity of the intensity map and showed in the same paper that the conjecture is true when $n=2$ and $n=3$. This conjecture is now known as the $4n-4$ conjecture. The sufficiency was proved by Fickus et al. in \cite{fickus2014phase}. One such ensemble consisting of $4n-4$ deterministic measurements was constructed in \cite{pohl2013phase} via a low rate sampling method. Unfortunately, the necessary part of the conjecture is false---an ensemble $\hat{\Phi}$ of $11$ $4$-dimensional measurement vectors whose intensity map $\mathcal{A}_{\hat{\Phi}}$ is injective was constructed in \cite{vinzant2015small}.


\subsection{Our Approach and Main Contributions}
Our work sets out with the interesting observation that the phase retrieval problem can be seen as a sensor network localization problem. More precisely, each component of the signal ${x}$ can be viewed as a point (which we will refer to as \emph{sensors}) in $d$-dimensional Euclidean space, where $d=1$ for real signals and $d=2$ for complex signals. If we explicitly design intensity measurements to form edges joining these sensors, then determining $x$ can be viewed as localizing the sensors in space. Furthermore, if the underlying graph generated by these measurements satisfies certain rigidity property, then each entry of $x$ can be uniquely determined up to the equivalence relation. In this work, we will design a deterministic measurement ensemble that fulfils the rigidity requirement and aim to construct an injective mapping with as few measurements as possible. 

Our contribution is threefold. First, we establish a connection between the phase retrieval problem and the well-studied sensor network localization problem. This allows us to use the tools from rigidity theory and graph realization to bear on the phase retrieval problem. In particular, we propose a two-stage algorithm to uniquely recover all phases up to the equivalence relation $\sim$ with few measurements. Concretely, we design a deterministic ensemble such that the underlying graph generated by intensity measurements is a $d$-lateration graph that is univerally rigid. By rigidity theory and relevant results in \cite{zhu2010universal}, we can easily obtain provable guarantee for unique recovery. For the non-sparse phase retrieval problem, our proposed ensemble consists of only $3n-2$ measurements and the corresponding intensity map is injective. For the sparse case where there are at most $s$ non-zero components, the number of measurements is further reduced to $n+2s-2$. Injectivity of our mapping is demonstrated by theoretical analysis. The algorithm is easy to implement and allows parallel computation. Simulations further demonstrate its efficacy and superiority over benchmark approaches in terms of efficiency. The computational complexity scales only linearly with $n$. To the best of our knowledge, this is the first work to study the phase retrieval problem by incorporating results from rigidity theory. Second, our ensemble design yields an injective intensity map of minimal size, and we provide explicit constants for the number of measurements. Last but not least, we propose a new variant of the phase retrieval problem and connect it to the complex rigidity theory proposed by Gortler and Thurston \cite{gortler2014generic}.

It should be pointed out that our approach does not constitute a counter-example to the necessary part of the $4n-4$ conjecture. In particular, the conjecture claims that the ensemble of any injective intensity map is of size at least $4n-4$. The injectivity is understood as a map on the whole $\mathbb{C}^n\!/\!\!\sim$, whereas our proposed algorithm provably recovers the correct signal with an additional minor assumption on the true signal. Nevertheless, as we will see in Section~\ref{sec:novelapp}, our method fails only for those signals that have its first two components collinear with the origin. We also remark that a different ensemble of the same size appeared in an unpublished manuscript \cite{pedarsani2014phasecode}. However, they did not provide motivations and insights for their ensemble, and the injectivity of the corresponding intensity map is not clear. 

\subsection{Organization}
The remainder of the paper is organized as follows. Section \ref{sec:snl} is devoted to revisiting the theory of sensor network localization and graph rigidity, which constitutes the fundamental basis for our approach. 
Section \ref{sec:novelapp} focuses on our novel approach to the phase retrieval problem, including a rigidity-theoretic two-stage algorithmic framework applied to the phase retrieval problem and a theoretical analysis to demonstrate the injectivity of the measurement ensemble generated by our algorithm. 
We then provide numerical results in Section \ref{sec:sim} to validate our theoretical findings, where we compare against three methods in the literature; namely, the Fienup algorithm~\cite{fienup1982phase}, the Wirtinger flow algorithm~\cite{candes2015phase}, and PhaseCut~\cite{waldspurger2015phase}. In Section \ref{sec:CVPR}, we study the complex rigidity theory and its connection to complex-measurement-based phase retrieval problems. Finally, we conclude the paper in Section \ref{sec:cln}.

Throughout the paper, the vectors are column vectors unless specified otherwise; ${e_j}$ denotes the $j$-th standard coordinate basis vector of suitable dimension; $(\cdot)^T$ and $(\cdot)^H$ denote the transpose and Hermitian transpose, respectively; $\langle\cdot,\cdot\rangle$ refers to the inner product of vectors; $\Re(\cdot)$ and $\Im(\cdot)$ denote the real and imaginary parts of a complex number or vector, respectively.

\section{Sensor Network Localization}\label{sec:snl}
In this section, we review the sensor network localization problem and a graph rigidity-theoretic approach to tackling it. Then, we present a novel connection between the phase retrieval problem and the sensor network localization problem. Such a connection allows us to utilize powerful results in rigidity theory to design the measurement ensemble that yields a minimal-size injective intensity map.

\subsection{Rigidity Theory and Sensor Network Localization}
The problem of sensor network localization is among the classic topics in signal processing and arises when one is interested in determining the positions of nodes in a network from a set of measurements. 
We do not attempt to give a comprehensive review, but only highlight the crucial findings of the relationship between unique localizability and graph rigidity theory. 

To begin, let us give a formal definition of the sensor network localization problem. Consider a network that consists of a number of \emph{anchor} nodes whose positions are known, together with a number of \emph{sensor} nodes whose locations are to be estimated. Let $d$ be the dimension of the Euclidean space in which these nodes reside. Let $G=(V,E)$ be the given network, where $V$ and $E$ denote the vertex set and the edge set of the graph, respectively. Without loss of generality, we assume that $G$ is connected. The vertices can be partitioned into two categories: the set $V_s=\{1,\ldots,n\}$ of sensors, and the set $V_a=\{n+1,\ldots,n+m\}$ of anchors. In particular, the positions of anchors are given by the vector $u\in\mathbb{R}^{dm}$. For the sake of clarity, we define three subsets of $E$, namely $E_{aa}$, $E_{sa}$, and $E_{ss}$, which are defined as $E_{aa}=\{(i,j)\in E: i,j\in V_a\}$, $E_{sa}=\{(i,j)\in E: i\in V_s, j\in V_a\}$, and $E_{ss}=\{(i,j)\in E: i,j\in V_s\}$, respectively. For $(i,j)\in E_{aa}$, the distances are trivially known; for $(i,j)\in E_{sa}$ or $(i,j)\in E_{ss}$, the distances are acquired by applying measurements. The distances between the nodes are represented by positive weights assigned to the edges, namely $r_{ij}$ for $(i,j)\in E_{sa}$ and $\tilde{r}_{ij}$ for $(i,j)\in E_{ss}$. For simplicity, we assume that all the measured data are noiseless. Let $r\in\mathbb{R}^{| E_{sa}|}$ and $\tilde{r}\in\mathbb{R}^{|E_{ss}|}$ be the collection of distance measurements. Then, an instance of the sensor network localization problem is given by $(G, (r,\tilde{r}),u, d)$. The objective is to find a position assignment $x\in\mathbb{R}^{nd}$ to the sensor nodes such that the following system is satisfied: 
\begin{align*}
\|x_i-x_j\|_2 &= r_{ij}, ~{\rm for}~(i,j)\in E_{sa};\\
\|x_i-x_j\|_2 &= \tilde{r}_{ij}, ~{\rm for}~(i,j)\in E_{ss};\\
x_i&\in\mathbb{R}^d , ~{\rm for}~i=1,\ldots,n.
\end{align*}
Herein, the pair $(x, u)$, which represents the positions of all nodes in space, is called a \emph{localization} of $G$. One interesting question in this setup is whether and when the sensor positions $x$ can be uniquely determined. If an instance admits a unique localization in $\mathbb{R}^d$, we say that it is \emph{uniquely localizable}. Eren et al. \cite{eren2004rigidity} utilized tools from rigidity theory to discuss the connection between unique localizability and properties of the associated network. In particular, they stated that following theorem.
\begin{thm}[Unique Localizability $\&$ Global Rigidity\cite{eren2004rigidity}]
For any $d\geq 1$, a generic sensor network localization instance is {uniquely localizable} if and only if its associated network $G=(V,E)$ is {globally rigid}. 
\end{thm}
In graph theory, a graph $G=(V,E)$ with $p$ being its localization in $\mathbb{R}^n$ is called \emph{globally rigid} if $p$ is the unique (up to congruence) localization of $G$ in $n$-dimensional Euclidean space. Nevertheless, it has been shown that even if an instance satisfies the global rigidity property, the problem of estimating the postions is still intractable in general~\cite{Saxe79}. To overcome the barrier, So and Ye introduced the notion of unique $d$-localizability in \cite{so2007theory}, while Zhu et al. applied the notion of universal rigidity to strengthen the connection\cite{zhu2010universal}. In particular, a generic sensor network localization instance is called \emph{uniquely $d$-localizable} if it admits a unique localization in \emph{any} Euclidean space with dimension $\ell \ge d$. A graph $G=(V,E)$ with $p$ being its localization in $\mathbb{R}^n$ is called \emph{universally rigid} if $p$ is the unique (up to congruence) localization of $G$ in \emph{any} Euclidean space. The connection between unique $d$-localizability and universal rigidity is presented below. For a rigorous proof, readers can refer to \cite[Theorem 2]{zhu2010universal}.
\begin{thm}[Unique $d$-Localizability $\&$ Universal Rigidity\cite{zhu2010universal} ]
For any $d\geq1$, a generic sensor network localization instance is {uniquely $d$-localizable} if and only if its associated network $G=(V,E)$ is {universally rigid}. 
\end{thm}
Although universal rigidity is more restrictive than global rigidity, it still captures a host of networks. Examples of universally rigid graphs include complete graphs and $d$-lateration graphs. The latter notion is defined as follows.
\begin{defi}[$d$-lateration Graph\cite{aspnes2006theory}]\label{def:trig}
Let $d,n\geq 1$ be integers with $n\geq d+1$. Then, an $n$-vertex graph $G=(V,E)$ is called a {$d$-lateration graph} if there exists an ordering $\{1,2,\ldots,n\}$ of the vertices in $V$ such that (i) the first $d+1$ vertices $1,2,\ldots, d$ form a complete graph; (ii) every vertex $j\geq d+1$ is connected to at least $d+1$ of the vertices $1,2,\ldots, j-1$. 
 \end{defi}
In particular, a sensor network localization instance is uniquely $d$-localizable if its associated network is a $d$-lateration graph. The proof is given in \cite[Theorem 3]{zhu2010universal}. Next, we will apply this proposition to design the measurement ensemble for phase retrieval.

\subsection{Sensor Network Localization and Phase Retrieval}
In this section, we consider the phase retrieval problem from a fresh perspective. In particular, we look at the problem through the lens of the sensor network localization problem. Concretely, each component of the signal ${x}$ can be regarded as a sensor in $d$-dimensional Euclidean space, where $d=1$ for real signals and $d=2$ for complex signals; while the origin can be viewed as an anchor. Since we assume that the measurement vectors can be designed freely, we restrict our attention to measurement vectors of the forms 
\begin{equation} \label{eq:measure-form}
\mbox{$\phi_k=e_k$ (where $1\le k\le n$) and $\tilde{\phi}_{jk}=e_j-e_k$ (where $1 \le j<k\le n$)}.
\end{equation}
The former yields the distance between sensor $k$ and the origin (i.e., $|x_k|$), while the latter yields the distance between sensors $j$ and $k$ (i.e., $|x_j-x_k|$). Thus, by choosing different subsets of measurement vectors from~\eqref{eq:measure-form}, we obtain different instances of the sensor network localization problem. Now, consider an instance of the sensor network localization problem constructed according to the above recipe, and let $G$ be the underlying graph. 
Based on results in the previous section, if $G$ is universally rigid, then the instance admits a unique localization in any Euclidean space. More precisely, all sensors are uniquely determined up to congruence in space, which implies unique recovery of the reconstructed signal $x$. We are thus motivated to construct an ensemble with as few measurement vectors from~\eqref{eq:measure-form} as possible, and yet the graph $G$ induced by these measurements is universally rigid. We will discuss how this can be achieved in the next section.
 
\section{A Rigidity-Theoretic Approach to Phase Retrieval}\label{sec:novelapp}
Let $x=(x_1,\ldots,x_n)$ be the signal vector that we wish to recover. To implement the idea in Section 2.2, we construct a graph $G=(V,E)$, where the vertex set $V$ is given by $V=\{0,x_1,\ldots,x_n\}$ (here, we use $x_i$ to denote both the label of the vertex and its location in space) and the edge set $E$ is obtained using the following procedure, so that $G$ is a $d$-lateration graph ($d=1$ if $x$ is a real signal and $d=2$ if $x$ is a complex signal).
\begin{enumerate}
  \item Choose $d+1$ nodes from $V$ as anchors and form a complete graph.
  \item Consider the remaining nodes as sensors. For each sensor node, construct $d+1$ edges connecting the sensor to all the anchors. 
\end{enumerate}
Since $G$ is universally rigid by construction, once the measurements corresponding to the edges of $G$ are available, the locations of the vertices are uniquely determined and so is the target signal vector. We now propose a rigidity-theoretic two-stage algorithm to actually recover the target signal vector.
We will first illustrate the idea for real signals and then extend it to complex ones. 
\subsection{Real Phase Retrieval}\label{subsubsec:real}
Consider the case where the signal $x=(x_1,\ldots,x_n)$ we wish to recover is real; i.e., $x_i \in \mathbb{R}$ for $i=1,\ldots,n$. In this case, we have $d=1$.
By Definition \ref{def:trig}, two anchors are required in order to construct the $1$-lateration graph. As we can specify the origin as an anchor, we only need to specify one more anchor. Towards that end, we measure the magnitude of each entry of $x$, thereby creating an edge between the origin and every other vertex. Let $j_1$ be the smallest index such that $|x_{j_1}| = w > 0$. Note that $x_{j_1}$ can be placed at either $w$ or $-w$. We fix the vertex $x_{j_1}$ at $w$ and specify it as an anchor. Next, we take the measurements $|x_i - x_{j_1}|$ for all $i \not= j_1$ and $|x_i| \not=0$, thereby creating an edge between $x_i$ and $x_{j_1}$. It is straightforward to verify from the definition that the resulting graph is a 1-lateration graph and hence is universally rigid. The target signal can then be recovered by simple calculations. The entire recovery procedure is summarized in Algorithm 1.
\begin{algorithm}[H]
\caption{\bf Real Phase Retrieval\label{alg1}}
  \begin{enumerate}
  \item Take the measurements $w_j=|\langle e_j,x \rangle|^2$ for $j=1,\dots,n$.
  \item Determine the indices $j_1,\dots,j_s$ of the non-zero entries and the sparsity $s$.
  \item Fix $x_{j_1}$ at $w_{j_1}$. 
  Treat the origin and $x_{j_1}$ as two anchors.
  \item For $k=2,\ldots,s$, take the measurements $w_{{j_k}1}=|\langle e_{j_k}-e_{j_1}, x\rangle |^2$ and solve for each $x_{j_k}$ by
      $$x_{j_k}=\frac{|\langle{ e}_{j_1},x\rangle|^2+|\langle{ e}_{j_k},x\rangle|^2-|\langle{ e}_{j_k}-{ e}_{j_1},x\rangle|^2}{2 x_{j_1}}.$$ \vspace{-5mm}
  \end{enumerate}
\end{algorithm}
For the non-sparse case, we can recover the signal up to reflection using the following $2n-1$ deterministic measurements: $$\Phi=\{{ e}_i\}_{i=1}^{n}\cup\{{ e}_j-{e}_1\}_{j=2}^{n}.$$
The size of our constructed ensemble coincides with the size that is necessary for successful recovery in the real case; 
see \cite[Theorem 4]{bandeira2014saving}. For the sparse case with sparsity $s$ ($s\leq n$), the deterministic ensemble
$$\Phi=\{{e}_i\}_{i=1}^{n}\cup \{{e}_{j_k}-{e}_{j_1}\}_{k=2}^{s}$$
yields an injective mapping with a size of $n+s-1$. Remarkably, the computational complexity of Algorithm \ref{alg1} is only $\Theta (n)$, which achieves the best order when compared with other methods in literature. 

\subsection{Complex Phase Retrieval}\label{subsubsec:cplx}
Since complex signal reconstruction is more common in practice, we now aim to extend Algorithm \ref{alg1} to the complex case. Recall that our task is to recover an $n$-dimensional complex signal from the measurements~\eqref{eq:measurement}. With the sensor network localization interpretation, complex phase retrieval amounts to localizing sensors on the plane; hence $d=2$ in this case. Three anchors, including the origin, are required to construct the $2$-lateration graph. A natural idea is to try the following direct extension of Algorithm \ref{alg1}.
\begin{algorithm}[H]
\caption{\bf Complex Phase Retrieval (Preliminary Idea)\label{alg2}}
  \begin{enumerate}
  \item Take the measurements $w_j=|\langle e_j,x \rangle|^2$ for $j=1,\dots,n$.
  \item Determine the indices $j_1,\dots,j_s$ of the non-zero entries and the sparsity $s$.
  \item Treat $x_{j_1},x_{j_2}$ together with the origin as three anchors. Localize $x_{j_1},x_{j_2}$ by another intensity measurement $z_1=|\langle e_{j_1}-e_{j_2},x \rangle|^2$.
  \item For $k=3,\dots,s$, take the measurements $w_{{j_k}1}=|\langle e_{j_k}-e_{j_1},x \rangle|^2$, $w_{{j_k}2}=|\langle e_{j_k}-e_{j_2},x \rangle|^2$ and solve for each $x_{j_k}$.
  \end{enumerate}
\end{algorithm}
Algorithm 2 yields a deterministic ensemble of small size and has a computational complexity of $\Theta(n)$.
However, we encounter a non-uniqueness issue when determining the artificial anchors. Recall that $x\sim y$ if and only if $y = e^{i\theta}x$ for some $\theta\in\mathbb{R}$. Let $\sim_w$ be the equivalence relation on $\mathbb{C}^n$ defined by 
$x \sim_w y$ if and only if $y=e^{\theta i}x $ or $ y=e^{\theta i}\bar{x}$ for some $\theta\in \mathbb{R}$. The equivalence relation $\sim$ captures \emph{isometry up to rotation}, whereas $\sim_w$ captures \emph{isometry up to rotation and reflection}. One may easily see that the artificial anchors can only achieve uniqueness up to rotation, but not both rotation and reflection. Therefore, even if the map $\mathcal{A}_{\Phi}^{w}:\mathbb{C}^n/\sim_w\rightarrow \mathbb{R}^M$ is injective, the map $\mathcal{A}_{\Phi}:\mathbb{C}^n/\sim\rightarrow \mathbb{R}^M$ is not guaranteed to be injective. Such deficiency cannot simply be resolved by adding more distance measurements. 

To tackle this issue, we now introduce another two measurements. The additional measurement vectors determine the relative phase between the two artificial anchors, thus eliminating the reflection ambiguity. The refined procedure is given in Algorithm \ref{alg3}. Our algorithm works under the following mild assumption.
\begin{ass}
 The first two non-zero entries of $x$ are not collinear with 0.
\end{ass}
\begin{algorithm}[H]
\caption{\bf Two-Stage Complex Phase Retrieval\label{alg3}}\vspace{2mm}
{\bf \ Stage 1: Building Artificial Anchors (without Reflection Ambiguity)}
  \begin{enumerate}
  \item Take the measurements $w_j=|\langle e_j,x \rangle|^2$ for $j=1,\dots,n$.
  \item Determine the indices $j_1,\dots,j_s$ of the non-zero entries and the sparsity $s$.
  \item Treat $x_{j_1},x_{j_2}$ together with the origin as three anchors. \textcolor{black}{Localize $x_{j_1},x_{j_2}$ up to rotation only by another two measurements
      $z_1=|\langle e_{j_1}+e_{j_2},x \rangle|^2$ and $z_2=|\langle e_{j_1}-ie_{j_2},x \rangle|^2$.}
  \end{enumerate}
  {\bf \ Stage 2: Localizing the Sensors}
  \begin{enumerate}
  \setcounter{enumi}{3}
  \item For $k=3,\dots,s$, take the measurements $w_{{j_k}1}=|\langle e_{j_k}-e_{j_1},x \rangle|^2$, $w_{{j_k}2}=|\langle e_{j_k}-e_{j_2},x \rangle|^2$ and solve the system \eqref{system_2} to recover $x_{j_k}$.
  \end{enumerate}
\end{algorithm}
Note that the ensemble used in Algorithm \ref{alg3} is $$\{e_j\}_{j=1}^n\cup\{e_{1}+e_{2},e_{1}-ie_{2}\}\cup\{e_{k}-e_{1}\}_{k=3}^{n}\cup\{e_{k}-e_{2}\}_{k=3}^{n}$$ for the non-sparse case, and is $$\{e_j\}_{j=1}^n\cup\{e_{j_1}+e_{j_2},e_{j_1}-ie_{j_2}\}\cup\{e_{j_k}-e_{j_1}\}_{k=3}^{s}\cup\{e_{j_k}-e_{j_2}\}_{k=3}^{s}$$ for the sparse case. The former has size $3n-2$, while the latter has size $n+2s-2$. 
Next, we will provide theoretical analysis to demonstrate the injectivity of our ensemble design. The proof of Theorem \ref{thm:inj} follows immediately from Lemmas \ref{lmm:stage1} and \ref{lmm:stage2}.
\begin{thm}[Unique Recovery of Algorithm \ref{alg3}]\label{thm:inj}
Suppose that the first two non-zero entries of $x$ are not collinear with 0. Then, $x$ can be exactly recovered by Algorithm~\ref{alg3} up to global phase; i.e., up to the equivalence relation $\sim$. 
\end{thm}
\begin{lem}[Stage 1 Correctness]\label{lmm:stage1}
Given the measurement ensemble $\Phi_1=\{e_j\}_{j=1}^n\cup\{e_{j_1}+e_{j_2},e_{j_1}-ie_{j_2}\}$ in Stage 1, the artificial anchors $x_{j_1}$ and $x_{j_2}$ are uniquely determined up to $\sim$.
\end{lem}

\begin{proof}
For simplicity, we assume that the first $s$ entries $x_1,\dots,x_s$ of $x$ are non-zero and the remaining entries are all zeros; i.e. $j_k=k$ for $k=1,\dots,s$. Our aim in this stage is to recover $x_1$ and $x_2$ up to a common phase shift. We achieve this by considering the ensemble $$\Phi=\left\lbrace\phi_1=e_1,\phi_2=e_2,\phi_3=\begin{pmatrix}1\\1\end{pmatrix},\phi_4=\begin{pmatrix}1\\-i\end{pmatrix}\right\rbrace.$$ The advantage of using this ensemble is twofold. First, it allows us to easily establish the injectivity of the induced intensity map $\mathcal{A}_{\Phi}:\mathbb{C}^2/\sim\rightarrow\mathbb{R}^4$. Second, with the measurements given by this ensemble, the reconstruction of $x_1$ and $x_2$ is almost trivial. To establish the injectivity of $\mathcal{A}_{\Phi}$, consider the so-called super-analysis operator $A_\Phi:\mathcal{H}^2\rightarrow\mathbb{R}^4$ given by $(AH)_j=\langle H,\phi_j\phi_j^*\rangle$ for $j=1,\dots,4$. It is easy to show that $$AH=\begin{pmatrix}H_{11}\\H_{22}\\H_{11}+H_{22}+H_{12}+H_{21}\\H_{11}+H_{22}+iH_{12}-iH_{21}\end{pmatrix}=\begin{pmatrix}H_{11}\\H_{22}\\H_{11}+H_{22}+2\Re{(H_{12})}\\H_{11}+H_{22}-2\Im{(H_{12})}\end{pmatrix}$$ and $AH=\textbf{0}\in\mathbb{R}^4$ if and only if $H$ is a zero matrix. Hence, $A_\Phi$ is injective. In particular, there is no matrix in the null space of $A_\Phi$ that is of rank 1 or 2. By a result of Bandeira et al. \cite{bandeira2014saving}, $\mathcal{A}_\Phi$ is injective. To recover $x_1$ and $x_2$, note that we have $w_1=|x_1|^2$, $w_2=|x_2|^2$, and
\begin{equation*}
\begin{cases} z_1 = |x_1+x_2|^2 = |x_1|^2+|x_2|^2+2\Re(x_1^*x_2),\\
z_2 = |x_1+ix_2|^2 = |x_1|^2+|x_2|^2-2\Im(x_1^*x_2). \end{cases}
\end{equation*}
Denoting $x_j=a_j+ib_j$ for $j=1,2$ and using the definition of $w_1$ and $w_2$, we have
\begin{equation*}
\begin{cases} z_1 = w_1+w_2+2(a_1a_2+b_1b_2),\\ z_2 = w_1+w_2-2(a_1b_2-a_2b_1). \end{cases}
\end{equation*}
Without loss of generality, we can assume that $x_1$ is a positive real; i.e. $a_1>0$ and $b_1=0$. Thus, we have
\begin{equation}\label{ab}
\begin{cases} a_1=\sqrt{w_1},\\b_1=0,\\ a_2=\frac{z_1-w_1-w_2}{2\sqrt{w_1}},\\ b_2=\frac{w_1+w_2-z_2}{2\sqrt{w_1}}. \end{cases}
\end{equation}
This proof is completed.\qed
\end{proof}

\begin{lem}[Stage 2 Correctness]\label{lmm:stage2}
Given the fixed anchors and the measurement ensemble $\Phi_2=\{e_j\}_{j=1}^n\cup\{e_{j_k}-e_{j_1}\}_{k=3}^{s}\cup\{e_{j_k}-e_{j_2}\}_{k=3}^{s}$ in Stage 2, the locations of the sensors $\{x_{j_k}\}_{k=3}^s$ are uniquely determined.
\end{lem}
\begin{proof}
Again, we assume that $x_1,\dots,x_s$ are the non-zero entries of $x$ and the remaining entries are all zeros. Let $x_1=a_1+ib_1$ and $x_2=a_2+ib_2$ be the anchors obtained in Stage 1 with $a_1,a_2,b_1,b_2$ defined by \eqref{ab}. Our goal is to uniquely determine $x_3,\dots,x_s$. Towards that end, recall that for $j=3,\ldots,s$, we have the measurements
\begin{equation}\label{system_2}
\begin{cases} w_{j}=|x_j|^2,\\w_{j1}=|x_j-x_1|^2, \\ w_{j2}=|x_j-x_2|^2. \end{cases}
\end{equation}
Denote $x_j=a_j+ib_j$ for $j=3,\dots,s$. Then,
\begin{equation*}
\begin{cases} w_{j1}=w_{j}+w_1-2\Re(x_j^*x_1)=w_{j}+w_1-2(a_1a_j+b_1b_j), \\ w_{j2}=w_{j}+w_2-2\Re(x_j^*x_2)=w_{j}+w_2-2(a_2a_j+b_2b_j). \end{cases}
\end{equation*}
In particular, for each $j$, we have the following system of 2 equations in 2 unknowns: 
\begin{equation}\label{system_1}
\begin{cases} a_1a_j+b_1b_j=\frac{1}{2}(w_1+w_{j}-w_{j1}), \\ a_2a_j+b_2b_j=\frac{1}{2}(w_2+w_{j}-w_{j2}). \end{cases}
\end{equation}
By assumption, the three points $0$, $x_1$, and $x_2$ are not collinear on $\mathbb{R}^2$. Therefore, we have $b_1/a_1\neq b_2/a_2 \Leftrightarrow a_1b_2-a_2b_1\neq 0$ and thus the solution to \eqref{system_1} is unique if it admits any solution at all. Since the true signal is feasible to this system, this completes the proof of the correctness of Stage 2.\qed
\end{proof}

\subsubsection{Remarks}
As we have pointed out in the Introduction, our algorithm does not provide a counter-example to the necessary part of the $4n-4$ conjecture since it requires an extra assumption that the first two entries of the target signal are not collinear with the origin. It is also worth noting that our method requires the least number of measurements known to date for both the real and complex scenarios and is still guaranteed to recover the target signal vector. Moreover, our method is extremely easy to implement.

\section{Simulation Results}\label{sec:sim}
In this section we provide numerical simulations to demonstrate the efficacy of our proposed algorithm. We compare our approach against three algorithms: Fienup Algorithm\cite{fienup1982phase}, Wirtinger Flow\cite{candes2015phase}, and PhaseCut\cite{waldspurger2015phase}. All simulations are implemented using MATLAB R2017a (version: 9.2.0.538062) on a computer running Windows 10 with an Intel i5-6000 CPU (with four 3.30 GHz processors) and 8 GB of main memory. The experimental setting is as follows. We uniformly sample $n$ complex-valued sensors within the square block $[-\frac{1}{2}, \frac{1}{2}] + i[-\frac{1}{2}, \frac{1}{2}]$. The number of measurements, $m$, is set to be $3n - 2$ for our approach, $6n$ for Fienup Algorithm, $4.5n$ for Wirtinger Flow, and $4n$ for PhaseCut. The maximum number of iterations is set to be $2,500$ for Fienup Algorithm, Wirtinger Flow, and PhaseCut. The step size of Fienup Algorithm is set to be $0.5$.
\begin{figure}[t]
 \vspace{-1.2cm}
  \centerline{\includegraphics[width=12cm]{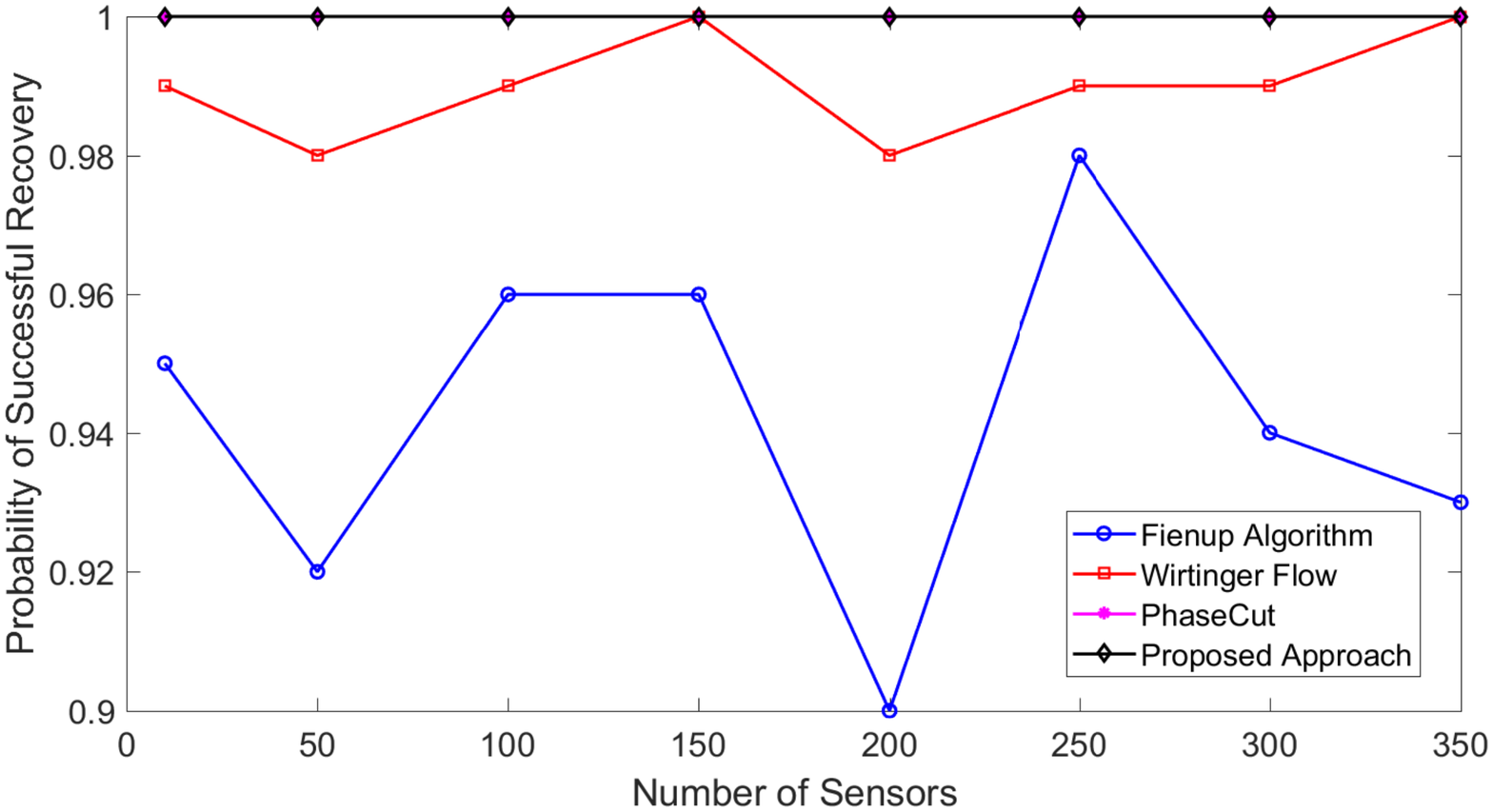}}
   \vspace{-1.2cm}
   \centering
  \caption{Probability of successful recovery in noiseless environment.}
	\label{fig:prob_success}
	\vspace{-1cm}
		\centering
	\includegraphics[width=12cm]{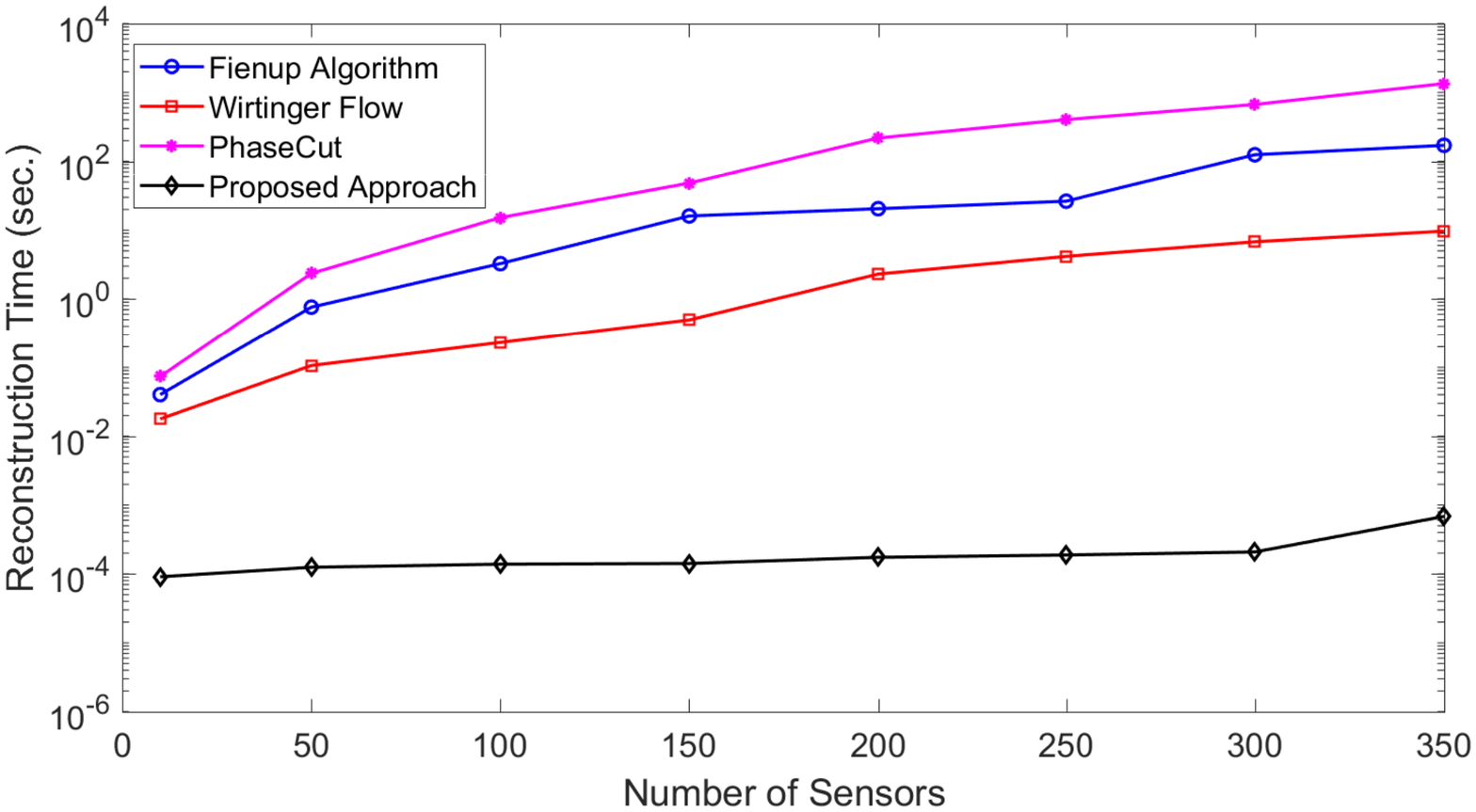}
	  \vspace{-1.2cm}
	  \centering
	\caption{Reconstruction time in noiseless environment.}
	\label{fig:efficiency}
\end{figure}

To investigate the effectiveness of our approach, we first compare the probability of successful recovery in the noiseless case. Specifically, we conduct the simulations using different number of sensors, where $n$ is specified in the set $\{10, 50, 100, 150, 200, 250, 300, 350\}$. For each $n$, we use 100 realizations of the sensors to obtain the figures. Figure~\ref{fig:prob_success} shows that both our approach and PhaseCut can recover all sensors. This is due to the closed-form solution in our approach and the refinement by using the Gerchberg-Saxton algorithm \cite{gerchberg1972practical} in PhaseCut. By contrast, neither the Fienup Algorithm nor Wirtinger Flow can recover all sensors as they require an initial guess of the unknown sensors. 

In Figure~\ref{fig:efficiency}, we report the average reconstruction time of each algorithm. The plot shows that our approach can recover the sensors substantially faster than other algorithms. This is again due to the closed-form solution in our approach. By contrast, other algorithms need to recursively update the approximated solutions until convergence or they terminate when the maximum number of iterations is reached. In particular, PhaseCut applies the Gerchberg-Saxton algorithm\cite{gerchberg1972practical} to refine the initial guess and lift up the dimension of the variables. The reconstruction time thus increases dramatically as the number of sensors increases. 
It is quite computationally expensive to reconstruct sensors by PhaseCut, even though it achieves nice recovery performance. 
\begin{figure}[h]
 \vspace{-1.2cm}
  \centerline{\includegraphics[width=12cm]{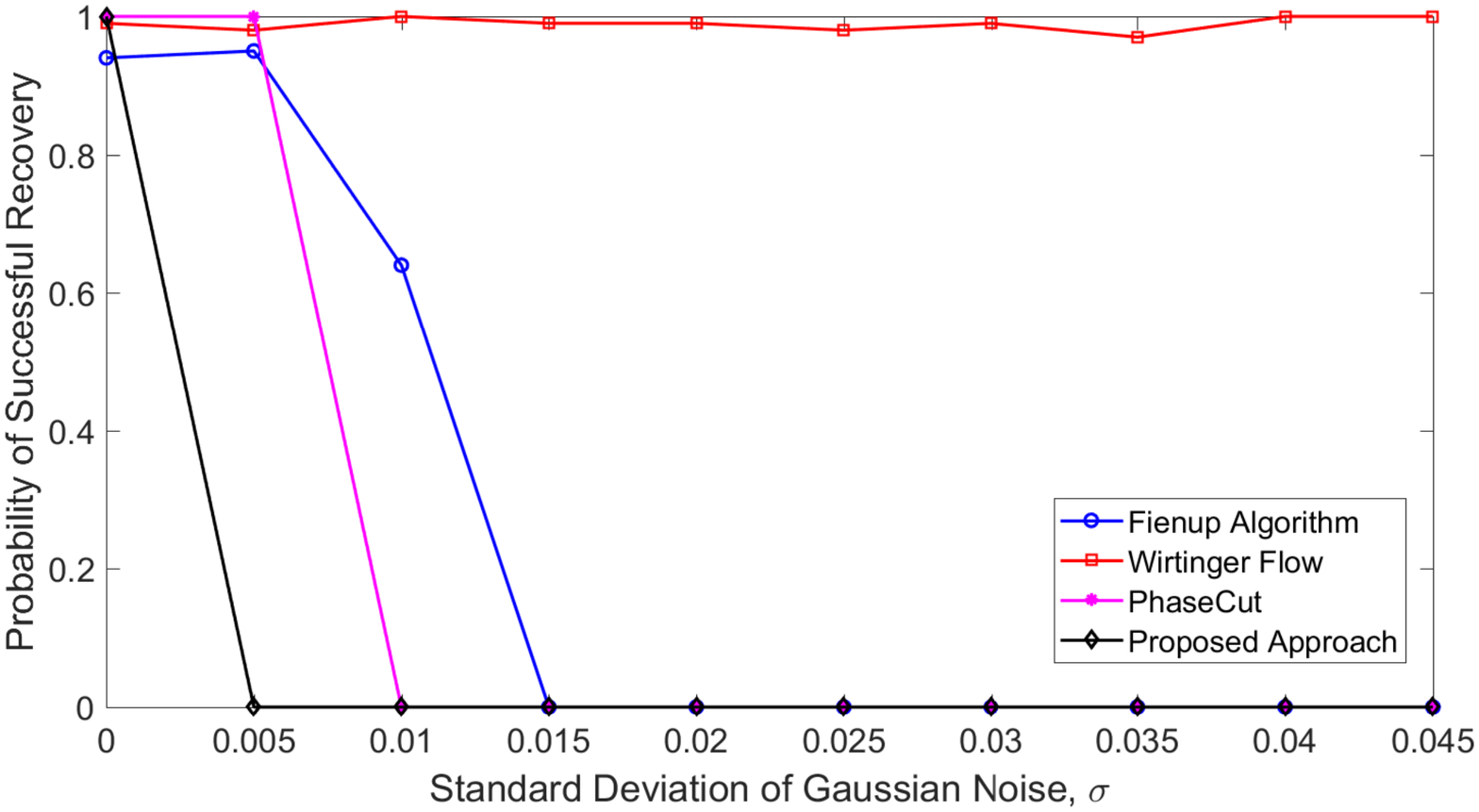}}
   \vspace{-1.2cm}
  \caption{Probability of successful recovery in noisy environment.}
	\label{fig:suc_prob_noisy}
\end{figure}
\begin{figure}[htb]
 \vspace{-1.2cm}
  \centerline{\includegraphics[width=12cm]{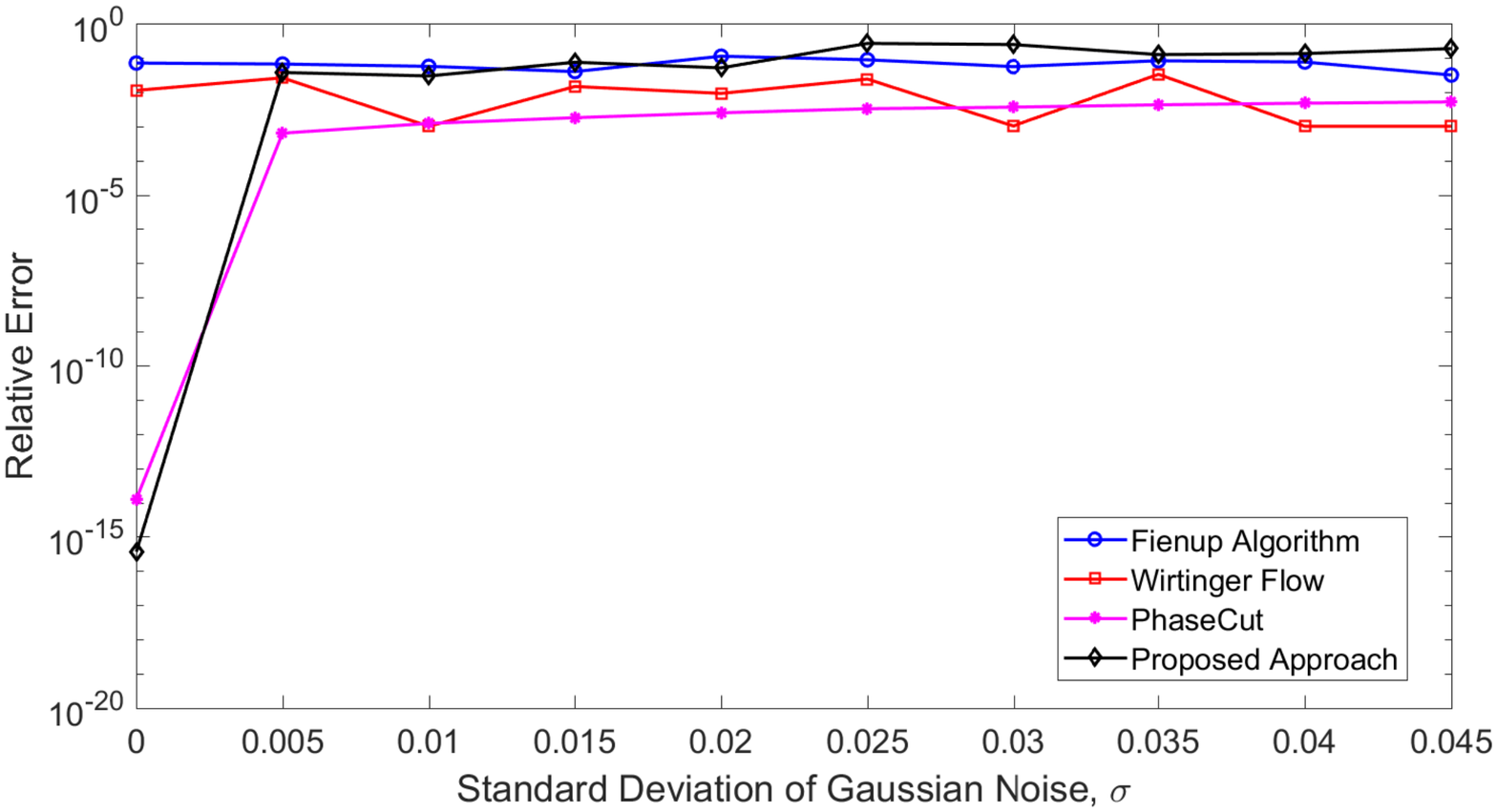}}
   \vspace{-1.2cm}
  \caption{Relative errors in noisy environment.}
	\label{fig:rel_err}
\end{figure}

In our experiments, we also test the robustness of our approach by adding noise to the measurements. Concretely, the measurement vector $b$ is of the form $b = |\mathcal{A}(x)|^2 + \epsilon(\sigma)$, 
where $\epsilon(\sigma) \in \mathbb{C}^m$ is a complex Gaussian white noise with standard deviation $\sigma$. The reconstruction error, $\delta$, is measured by the relative $\ell_2$-norm up to a complex phase:
\begin{equation*}
	\delta = \min_{\theta \in \mathbb{R} } {\frac{\lVert x - e^{i\theta}\hat{x} \rVert}{\lVert x \rVert}}.
\end{equation*}
We test the algorithms using different noise levels $\sigma \in \{0, 0.005, \ldots, 0.050\}$, where the number of sensors is set to be $100$. For each $\sigma$, we again use 100 realizations of the sensors. The probability of successful recovery and the relative errors in the noisy environment are reported in Figures \ref{fig:suc_prob_noisy} and~\ref{fig:rel_err}, respectively. From the plots, we find that our approach is less robust than other methods when the noise level exceeds $0.005$. The relative error of our approach is comparable with that of the Fienup Algorithm but higher than the other two methods.
The performance of our approach is not satisfactory in the noisy environment, as our approach is combinatorial in nature and hence the estimation error in each measurement tends to accumulate. A natural and interesting future direction is to robustify our approach while retaining its computational efficiency.

\section{Complex-Valued Phase Retrieval and Complex Rigidity Theory}\label{sec:CVPR}
In this section we will propose a new variant of the phase retrieval problem, which we call the \emph{complex-valued phase retrieval} (CVPR). We also discuss the connection between CVPR and complex rigidity theory (CRT)~\cite{gortler2014generic}.

\subsection{Complex Rigidity Theory}
We first briefly review the essential elements of complex rigidity theory. The definitions are taken from \cite{gortler2014generic}.
We equip the $d$-dimensional complex vector space $\mathbb{C}^d$ with the complex-valued distance $c(w,z)=\sum_{j=1}^d(w_j-z_j)^2$, where $w,z\in\mathbb{C}^d$. Note that the distance is a complex number in general. For a graph $G=(V,E)$, a \emph{configuration} of its vertices in $\mathbb{C}^d$ is a map $p:V\rightarrow\mathbb{C}^d$. The pair $(G,p)$ is called a \emph{framework}. Two frameworks $(G,p)$ and $(G,q)$ are said to be \emph{equivalent} if $c(p(i),p(j))=c(q(i),q(j))$ for all $(i,j)\in E$. Two configurations $p$ and $q$ are said to be \emph{congruent} if $c(p(i),p(j))=c(q(i),q(j))$ for all $i,j\in V$. A framework $(G,p)$ is said to be \emph{globally rigid} if for any framework $(G,q)$ equivalent to it, $p$ and $q$ are congruent. A configuration $p$ is said to be \emph{generic} if its coordinates do not satisfy any non-zero polynomial equation with rational coefficients (i.e., the coordinates of $p$ are \emph{algebraically independent}), and $(G,p)$ is called a \emph{generic framework}. A graph $G$ is said to be \emph{generically globally rigid} if all generic frameworks $(G,p)$ are globally rigid. It should be remarked that if we replace $\mathbb{C}$ by $\mathbb{R}$ above, then we are in the usual (real) rigidity theory setting. The following result by Gortler and Thurston [13] shows that when discussing the notion of generic global rigidity, there is no need to distinguish between the real and complex settings.
\begin{thm}\label{GGR}
A graph $G$ is generically globally rigid in $\mathbb{C}^d$ if and only if it is generically globally rigid in $\mathbb{R}^d$.
\end{thm}
For more discussions on CRT, readers are referred to \cite{gortler2014generic}.
\subsection{Complex-Valued Phase Retrieval and Its Connection to Complex Rigidity Theory}
As pointed out in Section \ref{subsubsec:cplx}, the connection between the complex phase retrieval problem and rigidity theory breaks down because they are concerned with different notions of symmetry. 
In this section, we restore this connection via \emph{the complex-valued intensity map $\mathcal{B}$}, which bears an even stronger resemblance with $\mathcal{A}:\mathbb{R}^n/\{\pm 1\}\limto\mathbb{R}^M$ than $\mathcal{A}:\mathbb{C}^n/\sim\limto\mathbb{R}^M$.

Let $\Phi=\{\phi_m\}_{m=1}^M$ be a complex ensemble. Then, the map $\mathcal{B}=\mathcal{B}_\Phi:\mathbb{C}^n/\{\pm 1\}\limto\mathbb{C}^M$ given by $\left(\mathcal{B}(x)\right)_m=\langle x,\phi_m \rangle^2$ is called \emph{the complex-valued intensity map induced by $\Phi$}. This map was first considered in \cite[Lemma 6]{bandeira2014saving}, where the authors showed that the injectivity of $\mathcal{A}:\mathbb{C}^n/\sim\limto\mathbb{R}^M$ implies the injectivity of $\mathcal{B}$. Note that this map can be easily realized physically and hence might be of practical importance \cite{basu2006distributed,zhang2012localization}. Since each observation is a complex number, we call the recovery of $x$ from the measurement map $\mathcal{B}$ \emph{the complex-valued phase retrieval} (CVPR).

In the sequel, we will always consider graphs with vertex set $V=\{0,1,\dots,n\}$ and measurement vectors from the set $\{e_i\}_{i=1}^n\cup\{e_i-e_j\}_{1\leq i<j\leq n}$. Every ensemble $\Phi=\{\phi_m\}_{m=1}^M\subseteq \{e_i\}_{i=1}^n\cup\{e_i-e_j\}_{1\leq i<j\leq n}$ from this set of measurement vectors defines an edge set $E_\Phi=E_{1,\Phi}\cup E_{2,\Phi}$ on the vertex set $V$, where $(i,0)\in E_{1,\Phi}$ if and only if $e_i\in\Phi$ and $(i,j)\in E_{2,\Phi}$ if and only if $e_i-e_j\in \Phi$. This gives rise to a graph $G_\Phi=(V,E_\Phi)$. We will omit the subscript $\Phi$ from the above notations when it causes no ambiguities. For any $z\in\mathbb{C}^n$, we define \emph{the configuration of $V$ induced by $z$} to be the map $p_z:V\rightarrow \mathbb{C}$ given by $p_z(i)=z_i$, where $i=1,\dots,n$ and $p_z(0)=0$.

\begin{lem}\label{lemma1}
The following statements hold.
\begin{enumerate}
\item[(i)] A vector $x\in\mathbb{C}^n$ is generic if its induced configuration $p_x$ is generic.\label{lemma_i}
\item[(ii)] If $p$ is a generic configuration of $V$ on $\mathbb{C}$, then the vector $x$ defined by $x_i=p(i)-p(0)$ is generic.\label{lemma1_ii}
\end{enumerate}
\end{lem}

\begin{proof}
We first prove (i). Suppose that $x$ is not generic. Then, there exists a non-zero $n$-variate polynomial $f$ with rational coefficients such that $f(x)=0$. Consider the $(n+1)$-variate polynomial $\hat{f}$ defined by $\hat{f}(z_0,z_1,\dots,z_n)=z_0f(z_1,\dots,z_n)$. Then, the vector $z=(p_x(0),p_x(1),\dots,p_x(n))$ satisfies $\hat{f}(z)=0$. Since the coefficients of $\hat{f}$ are also rational, $p_x$ is not generic.

Next, we prove (ii). Suppose that $x$ is not generic. Then, there exists a non-zero $n$-variate polynomial $f$ with rational coefficients such that $f(p(1)-p(0),\dots,p(n)-p(0))=0$. Consider one term $c\prod_{j=1}^n\left(p(j)-p(0) \right)^{r_j}$ in this polynomial, where $c\in\mathbb{Q}\backslash\{0\}$, $r_j\in\mathbb{Z}_{\geq 0}$, and $j=1,\dots,n$. It is easy to see that each such term is a non-zero polynomial in $p(0),p(1),\dots,p(n)$ with rational coefficients. Since $f\not\equiv 0$, there exists at least one such term. Moveover, a sum of terms of this form is again a non-zero polynomial in $p(0),p(1),\dots,p(n)$ with rational coefficients. Hence, $p(0),p(1),\dots,p(n)$ are algebraically dependent and thus $p$ is not generic.\qed
\end{proof}
 
The following theorem establishes the connection between CVPR and CRT.
\begin{thm}\label{mainthm}
For any ensemble $\Phi$, $\mathcal{B}$ is injective if and only if $(G, p)$ is globally rigid for any configuration $p$, where $G=G_\Phi=(V,E_\Phi)$ is defined above.
\end{thm}
\begin{proof}
Suppose that $\mathcal{B}$ is injective; i.e., $\mathcal{B}(x)=\mathcal{B}(y)$ implies that $x=\pm y$. Let $(G,p)$ and $(G,q)$ be two equivalent frameworks. By definition, $c(p(i),p(j))=c(q(i),q(j))$ for all $(i,j)\in E$. Define $x,y\in\mathbb{C}^n$ by $x_i=p(i)-p(0)$ and $y_i=q(i)-q(0)$ for $i=1,\dots,n$. Then, for any $(i,0)\in E_1$,
\begin{align*}
\langle e_i,x\rangle^2 &= x_i^2=(p(i)-p(0))^2=c(p(i),p(0))\\
&=c(q(i),q(0))=(q(i)-q(0))^2=y_i^2=\langle e_i,y\rangle^2,
\end{align*} and for any $(i,j)\in E_2$,
\begin{align*}
\langle e_i-e_j,x\rangle^2 &= (x_i-x_j)^2=(p(i)-p(j))^2=c(p(i),p(j))\\
&=c(q(i),q(j))=(q(i)-q(j))^2=(y_i-y_j)^2=\langle e_i-e_j,y\rangle^2.
\end{align*}
This is exactly $\mathcal{B}(x)=\mathcal{B}(y)$. Hence, $x=\pm y$ and
\begin{equation}\label{system1}
\begin{cases}
x_i^2=y_i^2,~i=1,\dots,n,\\
(x_i-x_j)^2=(y_i-y_j)^2,~i,j=1,\dots,n.
\end{cases}
\end{equation}
Thus, we have $c(p(i),p(j))=c(q(i),q(j))$ for all $(i,j)\in V$ and $p$ is congruent to $q$.

Now, suppose that $(G,p)$ is globally rigid for any configuration $p$. Let $x,y\in\mathbb{C}^n$ be such that $\mathcal{B}(x)=\mathcal{B}(y)$; i.e., $\langle x,\phi_m\rangle^2=\langle y,\phi_m\rangle^2$ for $m=1,\dots,M$. By the choice of $\{\phi_m\}_{m=1}^M$, 
\begin{align*}
& \begin{cases}
\langle e_i,x\rangle^2=\langle e_i,y\rangle^2,~i=1,\dots,n,\\
\langle e_i-e_j,x\rangle^2=\langle e_i-e_j,y\rangle^2,~i,j=1,\dots,n
\end{cases}\hspace{-3mm} \\
\Longleftrightarrow
& \begin{cases}
c(x_i,0)=c(y_i,0),~(i,0)\in E_1,\\
c(x_i,x_j)=c(y_i,y_j),~(i,j)\in E_2
\end{cases}.
\end{align*}
Therefore, $c(p_x(i),p_x(j))=c(p_y(i),p_y(j))$ for all $(i,j)\in E$; i.e., $(G,p_x)$ and $(G,p_y)$ are equivalent. By the supposition, they are also both globally rigid, and hence $p_x$ and $p_y$ are congruent. In particular, we have $c(p_x(i),p_x(j))=c(p_y(i),p_y(j))$ for all $i,j=1,\dots,n$, and the system \eqref{system1} is satisfied. From the first equation of \eqref{system1}, we know that $x_i=0$ if and only if $y_i=0$. For $x_i,x_j\neq 0$, we have $y_i,y_j\neq 0$ and $x_i/y_i=x_j/y_j$. Thus, $x_i/y_i$ is a complex constant $re^{i\psi}$ for those non-zero entries. Again from the first equation of \eqref{system1}, we know that $r=1$ and $2\psi=0\mod 2\pi$ since $\arg(x_i^2)=\arg(y_i^2)$. Thus, $\psi=0$ or $\psi=\pi$ and $x=\pm y$.\qed
\end{proof}

We say that a point \emph{$[x]\in\mathbb{C}^n/\{\pm 1\}$ is generic} if $x$ is generic. Since $x$ is generic if and only if $-x$ is generic, the above notion is well-defined. The following theorem should be compared with \cite[Theorem 2.9 and Theorem 3.4]{balan2006signal}; see also \cite[Section 3]{fickus2014phase}.
\begin{thm}
If $\mathcal{B}$ is injective on the set of generic points $[x]\in\mathbb{C}^n/\{\pm 1\}$ (i.e., for any generic $x$, $\mathcal{B}^{-1}\left(\mathcal{B}(x) \right)$ is the singleton $[x]=\{\pm x\}\in\mathbb{C}^n/\{\pm1 \}$), then $G$ is generically globally rigid on $\mathbb{C}$, and hence also on $\mathbb{R}$.
\end{thm}

\begin{proof}
Combining Theorem \ref{mainthm}, Lemma \ref{lemma1}(ii), and Theorem \ref{GGR}, the result follows immediately. \qed
\end{proof}

\section{Conclusion}\label{sec:cln}
In this paper, we studied the phase retrieval problem from a fresh perspective and connected it to the well-studied problem of sensor network localization. Based on this connection, we develop a rigidity-theoretic two-stage algorithm for phase retrieval that provably recovers the true signal using only $3n-2$ intensity measurements. Besides, our algorithm is efficient (its computational complexity is $\Theta(n)$) and easy to implement. Finally, we proposed a new variant of the phase retrieval problem and discuss its connection to complex rigidity theory. Adapting our approach to Fourier settings and extending our algorithm to the noisy case are definitely important topics for future research.

\newpage
\bibliographystyle{spmpsci} 
\bibliography{ref_PR}


\end{document}